\documentclass[10pt]{amsart}

\usepackage[cp1250]{inputenc}
\usepackage{amsmath}
\usepackage{amsfonts}
\usepackage{amsthm}
\usepackage{amsbsy}
\usepackage{multirow}

\newtheorem{Theorem}{Theorem}[section]

\newtheorem{Proposition}{Proposition}[section]

\theoremstyle{remark}
\newtheorem{Remark}{Remark}[section]

\theoremstyle{definition}
\newtheorem{Definition}{Definition}[section]
\newtheorem{Example}{Example}[section]

\begin{document}



\title[Space-convergent MCE]{A SPACE-CONSISTENT VERSION OF THE MINIMUM-CONTRAST ESTIMATOR FOR LINEAR STOCHASTIC EVOLUTION EQUATIONS}

\author{Pavel K\v r\' i\v z}

\address{University of Chemistry and Technology,  Prague, Department of Mathematics, Technick\' a 5, Prague 6, Czech Republic}
\email{pavel.kriz@vscht.cz}

\date{}




\begin{abstract}
A new modification of the minimum-contrast estimator (the weighted MCE) of drift parameter in a linear stochastic evolution equation with additive fractional noise is introduced in the setting of the spectral approach (Fourier coordinates of the solution are observed). The reweighing technique, which utilizes the self-similarity property, achieves strong consistency and asymptotic normality of the estimator as number of coordinates increases and time horizon is fixed (the space consistency). In this respect, this modification outperforms the standard (non-weighted) minimum-contrast estimator. Compared to other drift estimators studied within spectral approach (eg. maximum likelihood, trajectory fitting), the weighted MCE is rather universal. It covers discrete time as well as continuous time observations and it is applicable to processes with any value of Hurst index $H \in (0,1)$. To the author's best knowledge, this is so far the first space-consistent estimator studied for $H < 1/2$.
\end{abstract}

\maketitle


\section{Introduction}

This paper is a contribution to the spectral approach in the theory of statistical inference for parabolic linear stochastic partial differential equations (SPDEs), or more generally to linear stochastic evolution equations (SEE), with additive noise generated by a fractional Brownian motion (fBm). Coordinate projections of the solutions to these equations can be interpreted as real-valued fractional Ornstein-Uhlebneck processes (fOU). For more details on the spectral approach, consult the papers \cite{Lototsky2009} or more recent \cite{Cialenco2018}. With respect to the drift parameter estimation in linear SPDEs with additive noise, the following techniques have been studied:
\begin{itemize}
    \item The maximum likelihood estimators (MLE), initiated in \cite{Huebner-Rozovskii1995} for diagonalizable SPDEs driven by a cylindrical Wiener process and generalized for a cylindrical fBm with Hurst parameter $H \geq \frac{1}{2}$ in \cite{Cialenco-Lototsky-Pospisil2009}.
    \item The minimum contrast estimators (MCE), introduced in \cite{Kosky-Loges} for linear SPDEs with Wiener noise and studied in \cite{KrizMaslowski} and \cite{MaPo-Ergo} for equations driven by a fBm.
    \item The least squares estimator (LSE), application of which to one-dimensional projections of solutions to linear SPDEs driven by regular fBm was studied in \cite{Maslowski-Tudor2013}.
    \item The trajectory fitting estimator (TFE), introduced in the setting of parabolic diagonalizable linear SPDEs with Wiener noise in \cite{Cialenco-Gong-Huang2018}.
\end{itemize}

Properties of the MCE for drift parameter of a real-valued fractional Ornstein-Uhlenbeck process have been intensively studied in last few years. We refer the reader to the articles \cite{Sottinen-Vitasaary} for continuous-time setting, \cite{SebaiyViens2016} for discrete-time setting and \cite{Hu-Nualart-Zhou2017} for comparison with the LSE, to name just a few. These works benefit from the relation of Malliavin calculus and central limit theorems -- a popular theory initiated in \cite{Nualart-Pecatti2005} and further developed by many authors (see e.g. \cite{NourdinPeccati2015} and references therein). These techniques were recently applied to the MCE in infinite-dimensional setting in \cite{KrizMaslowski} and are also utilized in the present work. Another interesting work on minimum-contrast estimation is \cite{Bishwal2011}, where a different version of the MCE based on fundamental-martingale technique for a real-valued fOU with Hurst index $H>1/2$ is studied in the frameworks of both continuous-time and discrete-time sampling. Regarding drift parameter estimation for singular real-valued fOU processes ($H < 1/2$), recall the paper \cite{Hu-Nualart-Zhou2017}, where the LSE and the MCE for continuous-time observation are studied and the work \cite{kubilius2015}, which investigates discrete-time version of the LSE.\\

A modification of the minimum-contrast estimator (the weighted MCE) is introduced in this paper. Its construction benefits from the self-similarity property and it turns the time-consistent MCE (consistent with increasing time horizon) into a space-consistent estimator (fixed time horizon and increasing number of Fourier coordinates), which provides the best attainable speed of convergence in discrete-time setting. We believe this approach is potentially applicable to other types of time-consistent estimators, such as the LSE, and for different types of models (but still having the self-similarity property). To the author's best knowledge, this approach is new even in the basic case of parabolic diagonalizable equations with white additive noise (in space and time).\\

As demonstrated below, we see the main advantage of the newly proposed weighted MCE (over the above listed types of estimators) in its universality. With straightforward modifications, it can be used both for continuous-time and discrete-time observations, stationary and non-stationary processes and for noise processes with all values of Hurst parameter $H\in (0,1)$. These different settings require modified treatments and so they are studied separately in this paper. To the author's best knowledge, this is so far the first work that provides a drift estimator that is consistent in space (fixed time horizon, increasing number of coordinates) for infinite-dimensional SEEs with singular fractional noise ($H < 1/2$). Other estimators within spectral approach have been studied assuming $H=1/2$ or $H \geq 1/2$. \\

The main limitation of the weighted MCE is the fact that it requires the knowledge of the Hurst index $H$ and the noise intensities (volatilities) $\sigma_k$ in coordinates $k=1,2,\dots$. Values of these parameters can be known a priory in some special cases, such as noise that is white in time ($H = 1/2$) and white in space ($\sigma_k = 1, \, k=1,2,\dots$). Otherwise, these have to be determined or estimated from the observations. If continuous trajectories are observed (or high-frequency data considered), the value of $H$ can be determined using one of many infill consistent estimators for real-valued processes applied to a single coordinate projection observed in a fixed time window. To this respect let us mention the estimator of $H$ based on empirical quadratic variations of a filtered process (cf. \cite{Istas-Lang1997}), its version for a single point projection of an infinite-dimensional process of Ornstien-Uhlenbeck type (cf. \cite{torres2014}), the more robust (to outliers) estimators based on sample quantiles or trimmed means of a filtered real-valued fOU process (cf. \cite{Coeurjolly2008}) or the estimators based on a wavelet transform of a partially observed real-valued fOU process (cf. \cite{gloter2007} or \cite{Rosenbaum2008}), but this list is by no means complete. Simultaneous consistent estimation of $\sigma_k$ and $H$  from the observed $k$-th coordinate in high-frequency setting can be made using the powers of the second order variations (see \cite{Berzin_et_al2014}, Chapter 3.3). For infinite-dimensional fOU with discrete time setting (trajectories observed in fixed time instants, number of coordinate projections is increasing), one can estimate $H$ or both $H$ and $\sigma_k$ independently from the drift parameter by one of the procedures mentioned above and plug these estimates into the weighted MCE of the drift parameter (similarly to the approach for estimating drift, diffusion and Hurst parameter for real-valued fOU presented in \cite{Brouste-Iacus2013}). This may, however, negatively affect asymptotic properties of the weighted MCE. The study of this effect is beyond the scope of this paper. Another interesting work on joint estimation of drift, diffusion and Hurst parameter for discretely observed real-valued fOU is \cite{barboza-viens2017}, where the generalized method of moments (GMM) is applied. It might be an option to try to modify this GMM by appropriate reweighing (similarly to the modification of the MCE below) to get consistent simultaneous estimates in infinite-dimensional setting. Such study is, however, outside the scope of this article and might be the direction of further research.\\

This paper is organized as follows. In section \ref{sec: init setting}, the setting for the weighted MCE is specified. The conditions for existence of a (distribution-valued) stationary solution to a linear SEE with fractional additive noise are formulated. An example with fractional heat equation is presented. In section \ref{sec: estimation stat}, the weighted MCE for stationary solutions is derived and its consistency and asymptotic normality in space are proved. Discrete-time observations and continuous-time observations are studied separately, because formulas for estimators as well as asymptotic properties are different. In section \ref{sec: estimation non-stat}, the behavior of the weighted MCE for non-stationary solutions is considered. Section \ref{sec: comparison} is devoted to the comparison of the weighted MCE to other estimators.

\section{Initial setting}\label{sec: init setting}
Consider a linear stochastic evolution equation in a separable Hilbert space $\mathcal{V}$, which is driven by a fractional Brownian motion:

\begin{align}
dX(t) &= \alpha A X(t) dt + \Phi dB^{H}(t), \label{eq: SPDE}\\
 X(0) & = X_0. \label{eq: init_Cond}
\end{align}

In this equation, $\alpha > 0$ is an unknown parameter, $A : Dom(A) \subset \mathcal{V} \to \mathcal{V}$ and $\Phi : Dom(\Phi) \subset \mathcal{V} \to \mathcal{V}$ are densely-defined self-adjoint linear operators and  $(B^{H}(t), t \in \mathbb{R})$ is a standard two-sided cylindrical fractional Brownian motion on $\mathcal{V}$ with Hurst parameter $H \in (0,1)$, defined on a suitable probability space $(\Omega, \mathcal{F}, P)$. Note that $\Phi$ need not be bounded. The initial condition $X_0$ is assumed to be a random variable with values in an interpolation space $\mathcal{V}^{\gamma}$ (to be specified below) for some $\gamma \in \mathbb{R}$.\\

Assume that the equation \eqref{eq: SPDE} is diagonalizable, i.e. there is an orthonormal basis $\{e_k\}_{k \in \mathbb{N}}$ of the space $\mathcal{V}$  consisting of common eigenfunctions of operators $A$ and $\Phi$:\\
\begin{align}
& Ae_k = - \theta_k e_k, \quad \text{ with } \theta_k > 0, \text{ and} \label{eq: A1}\\
& \Phi e_k = \sigma_k e_k, \quad \text{ with } \sigma_k > 0. \label{eq: A1_2}
\end{align}
The standard two-sided cylindrical fractional Brownian motion $(B^{H}(t), t \in \mathbb{R})$ on $\mathcal{V}$ can be understood in the weak sense as a functional acting on $\mathcal{V}$ with
\[
\langle B^{H}(t), e_k\rangle = \beta_k^{H}(t), \quad \text{for }k=1,2,\ldots
\]
where $(\beta_k^{H}(t), t \in \mathbb{R})$ are mutually independent real-valued standard fractional Brownian motions and  $\langle B^{H}(t), x\rangle$ is the evaluation of $B^{H}(t)$ at $x$ (see e.g. \cite{MaPo-Ergo} for more details). Note that in infinite-dimensional setting, $B^{H}(t)$ does not take values in  $\mathcal{V}$.\\

Following the standard construction of the solutions to diagonalizable stochastic parabolic equations (cf. \cite{Cialenco-Lototsky-Pospisil2009}), we introduce a scale of Hilbert spaces $\mathcal{V}^\gamma$ indexed by $\gamma \in \mathbb{R}$ (also called the interpolation spaces). Take the strictly positive operator $\Lambda = \sqrt{I - A}$. The powers of this operator are well-defined and

\begin{equation}\label{eq: Lambda spectrum}
\Lambda^\gamma e_k = (1+\theta_k)^{\gamma /2} e_k.
\end{equation}

For $\gamma > 0$, we set $\mathcal{V}^\gamma$ to be the domain of $\Lambda^\gamma$ with the graph norm $|.|_{\mathcal{V}^\gamma} = |\Lambda^\gamma .|_\mathcal{V}$. For $\gamma = 0$, we set $\mathcal{V}^0 = \mathcal{V}$. Finally, for $\gamma < 0$ we define $\mathcal{V}^\gamma$ as the completion of $\mathcal{V}$ with respect to the graph norm $|.|_{\mathcal{V}^\gamma} = |\Lambda^\gamma .|_\mathcal{V}$. The interpolation spaces can be represented via coordinate projections:
\[
\mathcal{V}^\gamma = \left\{ v = \sum_{k=1}^{\infty}v_k e_k : \sum_{k=1}^{\infty}(1+\theta_k)^{\gamma} v_k^2 < \infty  \right\}, \quad \forall \gamma \in \mathbb{R},
\]
with
\[
|v|_{\mathcal{V}^\gamma}^2 = \left|\sum_{k=1}^{\infty}v_k e_k\right|_{\mathcal{V}^\gamma}^2 =  \left|\sum_{k=1}^{\infty}(1+\theta_k)^{\gamma /2}v_k e_k\right|_{\mathcal{V}}^2 = \sum_{k=1}^{\infty}(1+\theta_k)^{\gamma} v_k^2.
\]
Recall that for $\gamma_1 < \gamma_2$ the space $\mathcal{V}^{\gamma_2}$ is continuously and densely embedded into $\mathcal{V}^{\gamma_1}$ and for any $\gamma > 0$, $\mathcal{V}^{-\gamma}$ is the dual of $\mathcal{V}^{\gamma}$ relative to the inner product in $\mathcal{V}$ with the dual pairing:
\[
\langle v_1 | v_2 \rangle_\gamma = \langle \Lambda^{-\gamma} v_1, \Lambda^{\gamma} v_2\rangle_{\mathcal{V}}, \quad v_1 \in \mathcal{V}^{-\gamma}, v_2 \in \mathcal{V}^{\gamma}.
\]
Note that $\{e_k\}_{k \in \mathbb{N}}$ is an orthogonal basis of $\mathcal{V}^{\gamma}$ for each $\gamma \in \mathbb{R}$ and for any $v = \sum_{k=1}^{\infty}v_k e_k \in \mathcal{V}^\gamma$, the coordinates can be reconstructed by dual pairing:
\[
v_k = \langle v | e_k \rangle_\gamma.
\]

\begin{Definition}\label{def: diag solution}
The solution to the diagonalizable stochastic equation \eqref{eq: SPDE} with initial condition \eqref{eq: init_Cond} is a process $(X(t): t \geq 0)$ with values in $\mathcal{V}^{\gamma}$ for some $\gamma \in \mathbb{R}$ and with the expansion
\begin{equation}\label{eq: def solution sum}
X(t) = \sum_{k=1}^{\infty} x_k(t) e_k,
\end{equation}
where
\begin{align}
&x_k(t) = x_k(0)  e^{- \alpha \theta_k t} + \int_{0}^{t} e^{- \alpha \theta_k (t-s)} \sigma_k d\beta_k^{H}(s), \label{eq: 1D fOU} \\
&x_k(0) = \langle X_0| e_k \rangle_{\gamma}, \label{eq: 1D fOU_2}
\end{align}
and the sum \eqref{eq: def solution sum} converges in $L_2(\Omega,\mathcal{V}^\gamma)$ sense for some $\gamma \in \mathbb{R}$.
\end{Definition}

For each $k \in \mathbb{N}$, setting the initial condition to $x_k(0) = \int_{-\infty}^{0} e^{ \alpha \theta_k s} \sigma_k d\beta_k^{H}(s)$ makes the solutions \eqref{eq: 1D fOU} stationary fractional Ornstein-Uhlenbeck processes. Denote these processes $(z_k(t): t \geq 0)$ and build a stationary solution to the original equation \eqref{eq: SPDE} as follows
\begin{equation} \label{eq: constr. stat solut}
Z(t) := \sum_{k=1}^{\infty} z_k(t) e_k, \quad t \geq 0,
\end{equation}
if the sum converges in $L_2(\Omega,\mathcal{V}^\gamma)$ sense.

\begin{Theorem}\label{thm: existence of stat solution}
Let
\begin{equation}\label{eq: stat condition}
    \sum_{k=1}^{\infty} \frac{\sigma_k^2}{(1+\theta_k)^{\gamma}} \left(1+ \frac{1}{\theta_k} \right)^{2H} < \infty
\end{equation}
for some $\gamma \in \mathbb{R}$. Then equation \eqref{eq: SPDE} admits a stationary solution $(Z(t): t \geq 0)$ given by \eqref{eq: constr. stat solut} and $Z(t) \in L_{2}(\Omega, \mathcal{V}^{2H-\gamma})$ for each $t \geq 0$.\\

In addition, if
\begin{equation}\label{eq: init_integrable}
X_0 \in L_{2}(\Omega, \mathcal{V}^{2H-\gamma}),
\end{equation}
the equation \eqref{eq: SPDE} with initial condition \eqref{eq: init_Cond} has a solution $(X(t): t \geq 0)$ with $X(t) \in L_{2}(\Omega, \mathcal{V}^{2H-\gamma})$ for each $t \geq 0$.\\
\end{Theorem}

\begin{proof}
Recall (cf. for example \cite{Hu-Nualart-Zhou2017})
\begin{equation}\label{eq: r_k(0)}
\mathbb{E} z_k(t)^2 = \frac{\sigma_k^2}{(\alpha \theta_k)^{2H}}H\Gamma(2H) =: r_k(0), \quad \forall t \geq 0.
\end{equation}
Consequently
\begin{equation}
\begin{aligned}[b]
&\mathbb{E} \left| \sum_{k=1}^{\infty} z_k(t) e_k \right|^2_{\mathcal{V}^{2H-\gamma}} = \mathbb{E}  \sum_{k=1}^{\infty} (1+\theta_k)^{2H-\gamma} z_k(t)^2\\
& = \sum_{k=1}^{\infty} (1+\theta_k)^{2H-\gamma}\frac{\sigma_k^2}{(\alpha \theta_k)^{2H}}H\Gamma(2H).
\end{aligned}
\end{equation}
The condition \eqref{eq: stat condition} then ensures the existence and integrability of the stationary solution. \\

For the solution with the initial condition \eqref{eq: init_Cond}, write
\[
x_k(t) = z_k(t) - e^{-\alpha \theta_k t} z_k(0) + e^{-\alpha \theta_k t} x_k(0).
\]
Thus,
\[
\mathbb{E} x_k(t)^2 \leq C_1  \frac{\sigma_k^2}{\theta_k^{2H}} + C_2 \mathbb{E} x_k(0)^2.
\]
To conclude the proof, calculate
\[
\mathbb{E} \left| \sum_{k=1}^{\infty} x_k(t) e_k \right|^2_{\mathcal{V}^{2H-\gamma}} \leq C_1 \sum_{k=1}^{\infty}  (1+\theta_k)^{2H-\gamma}\frac{\sigma_k^2}{\theta_k^{2H}} + C_2 \sum_{k=1}^{\infty} (1+\theta_k)^{2H-\gamma} \mathbb{E} x_k(0)^2.
\]
The first sum is finite due to \eqref{eq: stat condition} and the second sum due to \eqref{eq: init_integrable}.
\end{proof}

Note that in typical applications (eg. $A$ being a differential operator on a smooth bounded $d$-dimensional domain) $\inf_k \{\theta_k\} > 0$ and we can simplify the condition \eqref{eq: stat condition} to the form:
\[
  \sum_{k=1}^{\infty} \frac{\sigma_k^2}{(1+\theta_k)^{\gamma}} < \infty.
\]

\begin{Remark}\label{remark: 2 operators}
Some works on drift parameter estimation in linear SPDEs (eg. \cite{Huebner-Rozovskii1995}, \cite{Cialenco-Lototsky-Pospisil2009} or \cite{Cialenco-Gong-Huang2018}) consider an additional term in the drift operator:
\begin{equation} \label{eq: SPDE - 2 term drift}
\begin{aligned}[b]
dX(t) &= (\alpha A_0 + A_1) X(t) dt + \Phi dB^{H}(t),\\
\end{aligned}
\end{equation}
with the corresponding diagonality assumption being
\begin{align}
& A_0 e_k = - \theta_k e_k, \quad \text{ with } \theta_k > 0, \label{eq: A1 - 2 term drift}\\
& A_1 e_k = - \nu_k e_k, \quad \text{ with } \nu_k > 0, \text{ and} \label{eq: A1 - 2 term drift_2}\\
& \Phi e_k = \sigma_k e_k, \quad \text{ with } \sigma_k > 0 \label{eq: A1 - 2 term drift_3}.
\end{align}
In this case, the assertion of Theorem \ref{thm: existence of stat solution} remains valid with $\theta_k$ being replaced by $\alpha \theta_k + \nu_k$. However, as discussed in Remark \ref{MCE - 2 term drift} below, the minimum-contrast estimator of $\alpha$ can no longer be expressed by a closed analytic formula, but is defined as an implicit solution to a moment equation. For the sake of simplicity and clarity, we thus stick to the simpler equation \eqref{eq: SPDE}.
\end{Remark}

\begin{Example}\label{example: stoch heat eqn - existence}
Consider the following formal heat equation with distributed fractional noise and Dirichlet boundary condition:

\begin{align}
&\frac{\partial f}{\partial t}(t,u) = \alpha \; \Delta f(t,u) + \eta^H(t,u), \quad \text{ for } (t,u) \in \mathbb{R}_+ \times \mathcal{O},  \label{eq: stochastic heat eq - naive} \\
&f(t,u) = 0, \quad \text{ for } (t,u) \in \mathbb{R}_+ \times \partial \mathcal{O}, \label{eq: stochastic heat eq - naive_2}  \\
&f(0,u)=X_0, \quad \text{ for } u \in \mathcal{O} \label{eq: stochastic heat eq - naive_3} ,
\end{align}

where $\Delta$ is Laplace operator,  $\mathcal{O} \subset \mathbb{R}^d$ is a bounded domain with smooth boundary $\partial \mathcal{O}$, $\alpha >0$ is the unknown parameter (e.g. heat conductivity), $X_0 \in L^2(\mathcal{O})$ is a deterministic initial condition and $(\eta^H(t,u): t\geq 0, u \in \mathcal{O})$ is a noise, which is fractional in time with Hurst parameter $H\in(0,1)$ and white in space.  \\

To give this formal equation rigorous meaning, reformulate it as a stochastic evolution equation (see \eqref{eq: SPDE})
\begin{align*}
&dX(t) = \alpha A X(t) dt + \Phi dB^{H}(t),\\
&X(0) = X_0,
\end{align*}

where $\mathcal{V} = L^2(\mathcal{O})$, $X_0 \in L^2(\mathcal{O})$, $A = \Delta|_{Dom(A)}$ with $Dom(A) = H^2(\mathcal{O}) \cap H^1_0(\mathcal{O})$ is  Dirichlet Laplace operator defined on a standard Sobolev space (cf. \cite{Shubin2001}), $(B^{H}(t), t \geq 0)$ is a cylindrical fBm and $\Phi$ is identity operator.\\

This equation is diagonalizable with eigenfunctions $\{e_k\}_{k \in \mathbb{N}}$ of $A$, which form an orthonormal basis of $L^2(\mathcal{O})$. The corresponding eigenvalues can be arranged in a sequence meeting the following growth condition (cf. \cite{Shubin2001}):
\begin{equation}\label{eq: Laplace_eigen_growth}
 \theta_k \asymp k^{\frac{2}{d}},
\end{equation}
where $a_k \asymp b_k$ means that there exist constants $0< c \leq C < \infty$ so that $c b_k \leq a_k \leq C b_k$  for all $k=1,2,\ldots$. In view of \eqref{eq: Laplace_eigen_growth}, condition \eqref{eq: stat condition} is fulfilled if $\gamma > \frac{d}{2}$ and $\eqref{eq: init_integrable}$ holds with $2H-\gamma \leq  0$.  Hence, for any $t\geq 0$, we have the existence of the solution $X_t \in L_{2}(\Omega, \mathcal{V}^{\min\{2H-\gamma,0\}})$, with any $\gamma > \frac{d}{2}$. In particular, should $X_t \in L_{2}(\Omega, L^2(\mathcal{O}))$, the condition $H > \frac{d}{4}$ must be satisfied.\\
\end{Example}

\section{Estimation in stationary case}\label{sec: estimation stat}

In this section, the weighted minimu-contrast estimator of $\alpha$ for stationary solution is derived.

\subsection{Preliminaries}
Recall the following $4^{th}$ moment theorem (see e.g. \cite{NourdinPeccati2015} or references therein for details):

\begin{Proposition}\label{prop: 4th moment}
Consider an isonormal Gaussian process $\mathbb{X}$ on a separable Hilbert space $\mathcal{H}$. Let $(F_n: n \in \mathbb{N})$ be a sequence of random variables belonging to the \mbox{$q$-th} Wiener chaos of $\mathbb{X}$ with $\mathbb{E}F^2_n = 1$ and consider a normally distributed random variable $U \sim \mathcal{N}(0,1)$. Then
\[
d_{TV}(F_n,U) \leq \sqrt{\frac{4q-4}{3q}}\sqrt{\mathbb{E}F_n^4 - 3} = \sqrt{\frac{4q-4}{3q}}\sqrt{\kappa_4(F_n)},
\]
where $d_{TV}$ denotes the total-variation distance of measures (or distributions of random variables) and
$\kappa_4(F_n) = \mathbb{E}F_n^4 - 3$ is the $4^{th}$ cumulant of $F_n$.
\end{Proposition}

\subsection{Discrete-time observations}
First assume that the processes $z_k$ are observed in discrete time instants, for simplicity let $t=1,2,\ldots,n$. Recall that the minimum-contrast estimator (see \cite{KrizMaslowski} or \cite{MaPo-Ergo}) is based on the sample second moments, which take the following form in our setting:
\[
\frac{1}{n} \sum_{t=1}^{n} |Z(t)|_{\mathcal{V}}^2 = \frac{1}{n} \sum_{t=1}^{n} \sum_{k=1}^{\infty} z_k(t)^2 = \sum_{k=1}^{\infty} \frac{1}{n} \sum_{t=1}^{n}  z_k(t)^2.
\]
Moreover, self-similarity of fractional Brownian motion implies that the distributions (on the space of trajectories) of the following two processes are same:
\begin{equation}\label{eq: self-similarity}
    Law\biggl(z_k(t): t \in [0,T] \biggr) = Law\biggl(\frac{\sigma_k}{(\alpha \theta_k)^H}z(\alpha \theta_k t): t \in [0,T] \biggr), \quad \forall k \in \mathbb{N},
\end{equation}
where $(z(t),t\geq 0)$ is the canonical fractional Ornstein-Uhlenbeck processes that is the stationary solution to equation
\[
dz(t) = -z(t)dt + d\beta^{H}(t).
\]

Hence, the values of the processes $z_k$ are scaled by $\frac{\sigma_k}{(\alpha \theta_k)^H}$ and the speed of their evolution by $\alpha \theta_k$. To fully utilize the information about $\alpha$ carried by each $z_k$, offset the effect of different scales of values by appropriate weights. For finitely many coordinates $z_k(t), k=1,\ldots,N$ observed in finitely many time instants $t=1,\ldots,n$, define
\begin{equation}\label{eq: def Y_N}
Y_N := \frac{\sum_{k=1}^{N} \frac{1}{n} \sum_{t=1}^{n}  \left(\frac{\theta_k^H}{\sigma_k}z_k(t)\right)^2}{N H \Gamma(2H)} = \frac{\sum_{k=1}^{N} \frac{\theta_k^{2H}}{\sigma_k^2} \frac{1}{n} \sum_{t=1}^{n}  z_k(t)^2}{N H \Gamma(2H)}.
\end{equation}
Using \eqref{eq: r_k(0)},  simple calculation yields
\[
\mathbb{E} (Y_N) = \alpha^{-2H}.
\]
This motivates the definition of the weighted minimum-contrast estimator:
\begin{equation}\label{eq: def rew_MC stat discr}
\alpha^{*}_N := (Y_N)^{-\frac{1}{2H}} = \left(\frac{\sum_{k=1}^{N} \frac{\theta_k^{2H}}{\sigma_k^2} \frac{1}{n} \sum_{t=1}^{n}  z_k(t)^2}{N H \Gamma(2H)}\right)^{-\frac{1}{2H}}.
\end{equation}

The so-called space asymptotics (number of coordinates $N$ grows to infinity, number of time instants $n$ remains fixed) of the weighted MCE is specified in the following theorem.

\begin{Theorem}\label{Thm: space asympt - stat discr}
Let the condition \eqref{eq: stat condition} holds (stationary solution exists) and consider the weighted minimum-contrast estimator $\alpha^{*}_N$ defined for a stationary solution in \eqref{eq: def rew_MC stat discr}. This estimator is \textbf{strongly consistent} in space, i.e.
        \begin{equation}\label{eq: consist stat discr}
        \alpha^{*}_N \overset{N\to \infty}{\longrightarrow} \alpha \quad \text{a.s.},
        \end{equation}
and it is \textbf{asymptotically normal} in space, i.e.
        \begin{equation}\label{eq: AN stat discr}
        \frac{\alpha^{*}_N - \alpha}{\frac{\alpha^{1+2H}}{2H}\sqrt{\text{var}(Y_N)}} \overset{N \to \infty}{\longrightarrow} U \sim \mathcal{N}(0,1) \quad \text{in distribution},
        \end{equation}
with $\text{var}(Y_N) \asymp \frac{1}{N}$ for $N \to \infty$.\\

In addition, let $\lim_{k\to \infty} \theta_k = \infty$ (which is typical e.g. for differential operators on smooth bounded domains). Then we have the asymptotic formula with explicit variance:
        \begin{equation}\label{eq: AN stat discr with variance}
        \sqrt{N}\left(\alpha^{*}_N - \alpha\right) \overset{N \to \infty}{\longrightarrow} \tilde{U} \sim \mathcal{N}(0,\frac{\alpha^{2}}{2n H^2}) \quad \text{in distribution}.
        \end{equation}
\end{Theorem}

\begin{proof}
Let us start with the strong consistency. Write
\[
Y_N - \mathbb{E}Y_N = \frac{1}{H \Gamma(2H)N} \sum_{k=1}^{N} Q_k - \mathbb{E}Q_k,
\]
where $Q_k = \frac{\theta_k^{2H}}{\sigma_k^2}\frac{1}{n}\sum_{t=1}^{n}z_k(t)^2$. Denote $r_k(i) := \mathbb{E}z_k(t+i)z_k(t)$ and, in view of the Kolmogorov strong law of large numbers (denote SLLN, see e.g. \cite{Shiryaev1996} for details), calculate
\begin{equation}\label{eq: calc SLLN cond}
\begin{aligned}[b]
&\sum_{k=1}^{\infty}\frac{\text{var}(Q_k)}{k^2} = \sum_{k=1}^{\infty}\frac{\frac{\theta_k^{4H}}{\sigma_k^4}\frac{2}{n} \sum_{i=-(n-1)}^{n-1}\left(1-\frac{|i|}{n}\right)r_k(i)^2}{k^2} \\
& \leq \sum_{k=1}^{\infty}\frac{\frac{\theta_k^{4H}}{\sigma_k^4}\; 4 \; r_k(0)^2}{k^2} = \frac{4 H^2 \Gamma(2H)^2}{\alpha^{4H}}\sum_{k=1}^{\infty}\frac{1}{k^2} < \infty.
\end{aligned}
\end{equation}
SLLN thus implies
\[
Y_N - \mathbb{E}Y_N  \overset{N\to \infty}{\longrightarrow} 0 \quad \text{a.s.}.
\]
Since $\mathbb{E}Y_N = \alpha^{-2H}$, the strong consistency is now immediate:
\[
\alpha^{*}_N = (Y_N)^{-\frac{1}{2H}}\overset{N\to \infty}{\longrightarrow} (\alpha^{-2H})^{-\frac{1}{2H}} =  \alpha \quad \text{a.s..}
\]

To explore asymptotic behavior of $\text{var}(Y_N)$, start with calculation
\begin{equation}\label{eq: s_k stat disc}
s_k^2 := \text{var} \left( \frac{1}{n} \sum_{t=1}^{n}z_k(t)^2\right) = \frac{2}{n} \sum_{i=-(n-1)}^{n-1}\left(1-\frac{|i|}{n}\right)r_k(i)^2 \asymp \frac{\sigma_k^4}{ \theta_k^{4H}}.
\end{equation}
Consequently,
\[
\text{var}(Y_N) = \frac{1}{(NH\Gamma(2H))^2} \sum_{k=1}^{N} \frac{\theta_k^{4H}}{\sigma_k^4} s_k^2  \asymp \frac{1}{N}.
\]
Next step is to show the asymptotic normality of $Y_N$ using the $4^{th}$ moment theorem. Calculation of the corresponding $4^{th}$ cumulant benefits from the independence of the coordinates:
\begin{equation}
\begin{aligned}[b]
&\kappa_{4}\left(\frac{Y_N - \alpha^{-2H}}{\sqrt{\text{var}(Y_N)}}\right) = \frac{1}{(\text{var}(Y_N))^2} \frac{1}{(N H \Gamma(2H))^4} \; \sum_{k=1}^{N} \frac{\theta_k^{8H}}{\sigma_k^8} \frac{1}{n^4} \; \kappa_{4}\left(\sum_{t=1}^{n}z_k(t)^2 - r_k(0)\right).
\end{aligned}
\end{equation}
Next, use the upper bound for the $4^{th}$ cumulant derived in \cite{KrizMaslowski}. In particular, Eq. (23) therein applied to the 1-dimensional Gaussian processes $z_k$ (so that $|Q(i)|_{\mathcal{L}_2}$ is replaced by $|r_k(i)|$) yields
\[
\kappa_{4}\left(\sum_{t=1}^{n}z_k(t)^2 - r_k(0)\right) \leq n C_2 \left(\sum_{i=-(n-1)}^{n-1}|r_k(i)|^{\frac{4}{3}}\right)^3 \leq 8 n^4 C_2 \frac{\sigma_k^8}{(\alpha \theta_k)^{8H}}(H \Gamma(2H))^4.
\]
This results in
\begin{equation}\label{eq: 4 kumul discr diag}
\kappa_{4}\left(\frac{Y_N - \alpha^{-2H}}{\sqrt{\text{var}(Y_N)}}\right) \leq \frac{C}{N},
\end{equation}
for some constant $C$ independent of $N$.

Proposition \ref{prop: 4th moment} now provides the upper bound for the total-variation distance from the $\mathcal{N}(0,1)$-distributed random variable $U$:
\begin{equation}\label{eq: dTV Y discr diag}
d_{TV}\left(\frac{Y_N - \alpha^{-2H}}{\sqrt{\text{var}(Y_N)}}, U \right) \leq \frac{C}{\sqrt{N}}.
\end{equation}
Application of the well-known delta method to mapping $g(x) = x^{-\frac{1}{2H}}$ results in the asymptotic normality of $\alpha^{*}_N$.\\

Now assume $\lim_{k\to \infty} \theta_k = \infty$. Observe that the self-similarity property \eqref{eq: self-similarity} implies
\begin{equation}\label{eq:r_k via r}
r_k(t) = \frac{\sigma_k^2}{(\alpha \theta_k)^{2H}} r(\alpha \theta_k t),
\end{equation}
where
\[
r(t) = e^{-t}H\Gamma(2H) + e^{-t} H (2H-1) \int_0^t \int_{-\infty}^0 e^{s+r}(s-r)^{2H-2} dr ds
\]
stands for the auto-covariance function of the canonical Ornstein-Uhlenbeck process. We can utilize the calculations from the strong consistency and extend them by employing the self-similarity property as follows
\begin{equation}
\begin{aligned}[b]
&\text{var}(Y_N) = \frac{\sum_{k=1}^{N} \text{var}(Q_k)} { H^2 \Gamma(2H)^2 N^2} = \frac{\sum_{k=1}^{N} \frac{\theta_k^{4H}}{\sigma_k^4}\frac{2}{n} \sum_{i=-(n-1)}^{n-1}\left(1-\frac{|i|}{n}\right)r_k(i)^2} { H^2 \Gamma(2H)^2 N^2} \\
&= \frac{\sum_{k=1}^{N} \frac{\theta_k^{4H}}{\sigma_k^4}\frac{2}{n} \sum_{i=-(n-1)}^{n-1}\left(1-\frac{|i|}{n}\right)\frac{\sigma_k^4}{(\alpha \theta_k)^{4H}} r(\alpha \theta_k i)^2} { H^2 \Gamma(2H)^2 N^2}
\end{aligned}
\end{equation}
Since $r(0) = H \Gamma(2H)$ and $\lim_{t \to \infty} r(t) = 0$, we have
\begin{equation}\label{eq: consist constant}
\lim_{N \to \infty} N \text{var}(Y_N) = \frac{2}{n}\frac{1}{\alpha^{4H}}.
\end{equation}
Combining this relation with \eqref{eq: AN stat discr} yields the expression \eqref{eq: AN stat discr with variance} with explicit asymptotic variance.
\end{proof}

Note that the asymptotic normality in space holds for any $H \in (0,1)$. This contrasts the asymptotic normality in time ($n \to \infty$) of this type of estimators, which is violated for $H > \frac{3}{4}$ (see e.g. \cite{SebaiyViens2016}, \cite{Sottinen-Vitasaary} or \cite{Hu-Nualart-Zhou2017}) due to the strong long-range dependence.\\

Moreover, if $H=1/2$ the achieved speed of convergence of the weighted MCE $\frac{1}{\sqrt{N}}$ is the best possible speed for the drift estimator one can get having discrete-time data with fixed number of time instants  (same speed as asymptotically efficient MLE, see \cite{Piterbarg-Rozovskii1997} or Section \ref{sec: comparison} below). We conjecture that this speed is best possible even in the fractional case ($H\neq 1/2$), which follows from the Cramer-Rao bound, because the coordinates are independent and each carries same amount of information about the unknown parameter (in contrast to the continuous-time setting, where the dilatation of time occurs, cf. \eqref{eq: self-similarity for integrals}, which enables faster decay, cf. Theorem \ref{Thm: space asympt - stat cont}).

\begin{Remark}
It is possible to use Proposition 6.1 from \cite{KrizMaslowski} to get the Berry-Esseen bounds for $\alpha^{*}_N$ on compacts. In particular, for each $K>0$ there exists a constant $C_{K} > 0$ such that
        \begin{equation}\label{eq: BE stat discr}
        \sup_{z\in [-K,K]} \biggl|\mathbb{P}\biggl(\frac{\alpha^{*}_N - \alpha}{\frac{\alpha^{1+2H}}{2H}\sqrt{\text{var}(Y_N)}} \leq z\biggr) - \mathbb{P}\biggl(U \leq z\biggr) \biggr| \leq  C_{K} \frac{1}{\sqrt{N}}.
        \end{equation}
\end{Remark}

\begin{Remark}\label{rem: time asymptot discr stat}
In contrast to previous setting, fix now the number of observed coordinates $N$ and consider the time (long-span) asymptotics of $\alpha^{*}_{N}$ ($n \to \infty$, fixed time step). Observe that $\alpha^{*}_{N}$ can be considered as the (non-weighted) minimum-contrast estimator constructed from the fractional Ornstein-Uhlenbeck process $U^{(N)}(t) = \sum_{k=1}^{N} \frac{\theta_k^H}{\sigma_k}z_k(t) e_k$. Hence, we can directly use \cite{KrizMaslowski} to see that:
\begin{itemize}
	\item $\alpha^{*}_{N}$ are strongly consistent as $n \to \infty$.
	\item If $H<\frac{3}{4}$, $\alpha^{*}_{N}$ are asymptotically normal as $n \to \infty$ with $\text{var}(\alpha^{*}_{N} - \alpha) = O(\frac{1}{n})$.
\end{itemize}
\end{Remark}

\begin{Remark}\label{MCE - 2 term drift}
Recall the more general equation with two-term drift operator from Remark \ref{remark: 2 operators}. In this case, Eq. \eqref{eq: r_k(0)} becomes
\begin{equation}
\mathbb{E} z_k(t)^2 = \frac{\sigma_k^2}{(\alpha \theta_k + \nu_k)^{2H}}H\Gamma(2H)
\end{equation}
and we can define the weighted minimum-contrast estimator $\hat{\alpha}$ as the (unique) solution to the equation
\begin{equation}\label{eq: wMCE - 2 term drift}
\frac{\sum_{k=1}^{N} \frac{(\hat{\alpha} \theta_k + \nu_k)^{2H}}{\sigma_k^2} \frac{1}{n} \sum_{t=1}^{n}  z_k(t)^2}{N H \Gamma(2H)} = 1.
\end{equation}
For the sake of simplicity, detailed inspection of this estimator is beyond the scope of this paper.
\end{Remark}

\subsection{Continuous-time observations}
Observations of processes $z_k(t)$ in continuous time window $t \in [0,T]$ are considered in this section. Straightforward modification of the estimator \eqref{eq: def rew_MC stat discr} (substituting sums by integrals) would preserve all properties specified in Theorem \ref{Thm: space asympt - stat discr}. However, if $\theta_k \overset{k \to \infty}{\longrightarrow} \infty$ (a typical situation for differential operators on smooth bounded domains), we can further improve the estimator. Recall the self-similarity property \eqref{eq: self-similarity}:
\[
Law\biggl(z_k(t): t \in [0,T] \biggr) = Law\biggl(\frac{\sigma_k}{(\alpha \theta_k)^H}z(\alpha \theta_k t): t \in [0,T] \biggr).
\]
Change of variable leads to:
\begin{equation}\label{eq: self-similarity for integrals}
Law \biggl( \frac{1}{T} \int_{0}^{T}z_k(t)^2 dt \biggr) = Law\biggl(\frac{\sigma_k^2}{(\alpha \theta_k)^{2H}}  \frac{1}{\alpha \theta_k T} \int_{0}^{\alpha \theta_k T} z(t)^2 dt \biggr).
\end{equation}
Thus, increasing $\theta_k$ changes not only the scale of values, but also increases the time horizon of the process $z$ (understood in law), which is $\alpha \theta_k T$. To make use of this increasing time horizon (dilatation of time), weights should be growing faster compared to the discrete time case.\\

To derive appropriate weights, consider a weighted MCE with general weights constructed from
\begin{equation}\label{eq: rew_MC gen weights cont}
Y_N(w_1,..,w_N) := \frac{\sum_{k=1}^{N} w_k \frac{1}{T} \int_{0}^{T}z_k(t)^2 dt} {H \Gamma(2H)\sum_{k=1}^{N} w_k \frac{\sigma_k^2}{\theta_k^{2H}}}.
\end{equation}
Obviously $\mathbb{E} Y_N(w_1,..,w_N) = \alpha^{-2H}$. Set the weights $w_1,\ldots,w_N$ in order to minimize the variance
\begin{equation}\label{eq: var Y_N cont}
\text{var} (Y_N(w_1,..,w_N)) = \frac{\sum_{k=1}^{N} w_k^2 s_k^2} {\left( H \Gamma(2H)\sum_{k=1}^{N} w_k \frac{\sigma_k^2}{\theta_k^{2H}} \right)^2},
\end{equation}
where
\[
s_k^2 = \text{var} \left(\frac{1}{T} \int_{0}^{T}z_k(t)^2 dt \right).
\]
The optimum solution is
\[
w_k^{(\text{opt})} = \frac{\sigma_k^2}{\theta_k^{2H}} \frac{1}{s_k^2},
\]
and (using the self-similarity and change of variable)
\[
s_k^2 = \frac{\sigma_k^4}{(\alpha \theta_k)^{4H}} \frac{4}{\alpha \theta_k T} \int_{0}^{\alpha \theta_k T}r(s)^2\left(1-\frac{s}{\alpha \theta_k T} \right) ds,
\]
where $r(s) = \mathbb{E}z(s)z(0)$ is the auto-covariance function of the canonical fractional Ornstein-Uhlenbeck process. In \cite{Cheridito_et_al2003}, it is shown that
\[
r(s) = H(2H-1)s^{2H-2} + O(s^{2H-4}), \quad s \to \infty.
\]
Consequently, if  $\theta_k \overset{k \to \infty}{\longrightarrow} \infty$, dilatation of time results in different forms of the weighted MCEs $\alpha^{*}_N$ for the Hurst index $H \in (0,3/4)$, $H = 3/4$ and $H \in (3/4,1)$ due to different time asymptotics of $s^2_k$:
\begin{itemize}
	\item For $0<H<\frac{3}{4}$ we have $s_k^2 \asymp \frac{\sigma_k^4}{\theta_k^{4H+1}}$ as $k \to \infty$, and so we set $w_k = \frac{\theta_k^{2H+1}}{\sigma_k^2}$ and
	\begin{equation}\label{eq: def rew_MC stat cont I}
		\alpha^{*}_N := Y_N^{-\frac{1}{2H}} =  \left(\frac{\sum_{k=1}^{N} \frac{\theta_k^{2H+1}}{\sigma_k^2} \frac{1}{T} \int_{0}^{T}  z_k(t)^2 dt}{ H \Gamma(2H) \sum_{k=1}^{N} \theta_k }\right)^{-\frac{1}{2H}}  .
	\end{equation}

	\item For $H=\frac{3}{4}$ we have $s_k^2 \asymp \frac{\sigma_k^4}{\theta_k^{4}} \ln(\theta_k T)$ as $k \to \infty$, we define $w_k = \frac{\theta_k^{\frac{5}{2}}}{\sigma_k^2 \ln(\theta_k T)}$ and
	\begin{equation}\label{eq: def rew_MC stat cont II}
		\alpha^{*}_N := Y_N^{-\frac{1}{2H}} =  \left(\frac{\sum_{k=1}^{N} \frac{\theta_k^{\frac{5}{2}}}{\sigma_k^2 \ln(\theta_k T)} \frac{1}{T} \int_{0}^{T}  z_k(t)^2 dt}{ \frac{3}{4} \Gamma(\frac{3}{4}) \sum_{k=1}^{N} \frac{\theta_k}{\ln(\theta_k T)} }\right)^{-\frac{1}{2H}}.
	\end{equation}

	\item For $\frac{3}{4} < H < 1$ we have $s_k^2 \asymp \frac{\sigma_k^4}{\theta_k^{4}}$ as $k \to \infty$, we set $w_k = \frac{\theta_k^{4-2H}}{\sigma_k^2}$ and
	\begin{equation}\label{eq: def rew_MC stat cont III}
		\alpha^{*}_N := Y_N^{-\frac{1}{2H}} = \left(\frac{\sum_{k=1}^{N} \frac{\theta_k^{4-2H}}{\sigma_k^2} \frac{1}{T} \int_{0}^{T}  z_k(t)^2 dt}{ H \Gamma(2H) \sum_{k=1}^{N} \theta_k^{4-4H} }\right)^{-\frac{1}{2H}}.
	\end{equation}
\end{itemize}

\begin{Remark}
The corresponding optimization procedure applied in the discrete-time setting leads exactly to the estimator \eqref{eq: def rew_MC stat discr}.
\end{Remark}

\begin{Theorem}\label{Thm: space asympt - stat cont}
Assume \eqref{eq: stat condition} holds (i.e. stationary solution exists),  $\theta_k \overset{k \to \infty}{\longrightarrow} \infty$ and consider the weighted minimum-contrast estimators $\alpha^{*}_N$ defined for a stationary solution in \eqref{eq: def rew_MC stat cont I}, \eqref{eq: def rew_MC stat cont II} and \eqref{eq: def rew_MC stat cont III}. These estimators are \textbf{strongly consistent} in space, i.e.
        \begin{equation}\label{eq: consist stat cont}
        \alpha^{*}_N \overset{N\to \infty}{\longrightarrow} \alpha \quad \text{a.s.}.
        \end{equation}
In addition, assume
\begin{equation}\label{eq: assumption AS}
\begin{aligned}[b]
&\biggl(H \leq \frac{3}{4}\biggr), \quad \text{or} \quad \biggl(\frac{3}{4} < H <1 \text{ and } \theta_k \asymp k^\beta \text{ for some } \beta > 0 \biggr).
\end{aligned}
\end{equation}
Then $\alpha^{*}_N$ are \textbf{asymptotically normal}, i.e.
        \begin{equation}\label{eq: AN stat cont}
       \frac{\alpha^{*}_N - \alpha}{\frac{\alpha^{1+2H}}{2H}\sqrt{\text{var}(Y_N)}} \overset{N \to \infty}{\longrightarrow} U \sim \mathcal{N}(0,1) \quad \text{in distribution},
        \end{equation}
with
\begin{equation}\label{eq: Y_N asymp variance cont}
        \text{var}(Y_N) \asymp
	       \begin{cases}
        	\frac{1}{\sum_{k=1}^{N}\theta_k} \quad \text{for }0<H<\frac{3}{4},\\
	        \frac{1}{\sum_{k=1}^{N}\frac{\theta_k}{\ln(\theta_k T)}} \quad \text{for }H=\frac{3}{4},\\
            \frac{1}{\sum_{k=1}^{N}\theta_k^{4-4H}} \quad \text{for }\frac{3}{4}<H<1.
	       \end{cases}
        \end{equation}
\end{Theorem}

\begin{proof}
For strong consistency, apply SLLN to $Y_N$. In particular, for $H<\frac{3}{4}$, note that
\[
H \Gamma(2H) \sum_{k=1}^{N} \theta_k \nearrow \infty \text{ as }N \to \infty
\]
and
\begin{equation}
\begin{aligned}[b]
	&\sum_{k=1}^{\infty}\frac{\text{var}\left( \frac{\theta_k^{2H+1}}{\sigma_k^2} \frac{1}{T} \int_{0}^{T}  z_k(t)^2 dt \right)}{\left( H \Gamma(2H) \sum_{m=1}^{k} \theta_m \right)^2} = \sum_{k=1}^{\infty}\frac{ \frac{\theta_k^{4H+2}}{\sigma_k^4} s_k^2}{\left( H \Gamma(2H) \sum_{m=1}^{k} \theta_m \right)^2}\\
	&\leq C  \sum_{k=1}^{\infty}\frac{ \theta_k}{\left( \sum_{m=1}^{k} \theta_m \right)^2} <\infty,
\end{aligned}
\end{equation}
where the convergence of the last series follows from the fact that for $k\geq 2$:
\[
\frac{ \theta_k}{\left( \sum_{m=1}^{k} \theta_m \right)^2} \leq \frac{ \theta_k}{\left( \sum_{m=1}^{k} \theta_m \right)\left( \sum_{m=1}^{k-1} \theta_m \right)} = \frac{1}{\sum_{m=1}^{k-1} \theta_m} - \frac{1}{\sum_{m=1}^{k} \theta_m},
\]
which leads to the telescopic series
\[
\sum_{k=1}^{\infty}\frac{ \theta_k}{\left( \sum_{m=1}^{k} \theta_m \right)^2} \leq \frac{1}{\theta_1} + \left(\frac{1}{\theta_1} - \frac{1}{\theta_1+\theta_2} \right) + \left(\frac{1}{\theta_1+\theta_2} - \frac{1}{\theta_1+\theta_2 +\theta_3} \right)+ \ldots = \frac{2}{\theta_1}.
\]
This verifies the assumptions of the SLLN and the almost-sure convergence of $Y_N$ and strong consistency of $\alpha^{*}_N$ are guaranteed.\\

If $H=\frac{3}{4}$ or $\frac{3}{4}< H < 1$, strong consistency of $\alpha^{*}_N$ can be proved similarly, with $\theta_k$ being replaced with $\frac{\theta_k}{\ln(\theta_k T)}$ or with $\theta_k^{4-4H}$ in the conditions for SLLN above.\\

To show \eqref{eq: Y_N asymp variance cont}, combine Eq. \eqref{eq: var Y_N cont} with formulas for $w_k$ and asymptotic formulas for $s_k^2$, separately for the case $0<H<\frac{3}{4}$, $H=\frac{3}{4}$ and $\frac{3}{4} < H < 1$.\\

For asymptotic normality, start with calculations similar to the discrete-time case (using formula (27) from \cite{KrizMaslowski}):
\begin{equation}
\begin{aligned}[b]
&\kappa_{4}\left(\frac{Y_N - \alpha^{-2H}}{\sqrt{\text{var}(Y_N)}}\right) = \frac{1}{(\sum_{k=1}^{N} w_k^2 s_k^2)^2}
\sum_{k=1}^{N} w_k^4 \; \kappa_{4} \left(\frac{1}{T} \int_{0}^{T} z_k(t)^2 dt - r_k(0) \right) \\
&\leq \frac{1}{(\sum_{k=1}^{N} w_k^2 s_k^2)^2}
\sum_{k=1}^{N} w_k^4 \frac{\tilde{C}}{T^3} \left(\int_{-T}^{T} r_k(t)^{\frac{4}{3}} dt \right)^3.
\end{aligned}
\end{equation}
where $\tilde{C}$ is a universal constant. To proceed further, we use the formula \eqref{eq:r_k via r}, the change-of-variable formula and the upper bound for the covariance function of the canonical fractional Ornstein-Uhlenbeck process (see e.g. Lemma 5.2 in \cite{KrizMaslowski}):
\[
|r(t)| \leq \min\{r(0), C \; |t|^{2H-2}\}.
\]
Calculations of the integrals of the resulting power functions then lead to the following upper bounds:
\[
\kappa_{4}\left(\frac{Y_N - \alpha^{-2H}}{\sqrt{\text{var}(Y_N)}}\right) \leq
	\begin{cases}
	 \frac{C_1}{(\sum_{k=1}^{N} w_k^2 s_k^2)^2} \sum_{k=1}^{N}  \frac{w_k^4}{T^3} \frac{\sigma_k^8}{(\alpha \theta_k)^{8H+3}}    \quad \text{for }0<H<\frac{5}{8},\\[10pt]
	 \frac{C_2}{(\sum_{k=1}^{N} w_k^2 s_k^2)^2} \sum_{k=1}^{N}  \frac{w_k^4}{T^3} \frac{\sigma_k^8}{(\alpha \theta_k)^{8H+3}}\ln^3(\alpha \theta_k T) \quad \text{for }H=\frac{5}{8},\\[10pt]
	 \frac{C_3}{(\sum_{k=1}^{N} w_k^2 s_k^2)^2} \sum_{k=1}^{N}  \frac{w_k^4}{T^3} \frac{\sigma_k^8}{(\alpha \theta_k)^{8H+3}} (\alpha \theta_k T)^{8H - 5} \quad \text{for }\frac{5}{8}<H<1.
	\end{cases}
\]
If we combine these bounds with the corresponding formulas for $w_k$ and asymptotic formulas for $s_k^2$, we obtain
\[
\kappa_{4}\left(\frac{Y_N - \alpha^{-2H}}{\sqrt{\text{var}(Y_N)}}\right) \leq C \; \zeta(N),
\]
with
\begin{equation}\label{eq: def zeta(N)}
            \zeta(N) =
        	\begin{cases}
	           \frac{1}{T^3 \; \alpha^{8H+3}} \frac{1} {\sum_{k=1}^{N} \theta_k}    \quad \text{for }0<H<\frac{5}{8},\\[10pt]
	           \frac{1}{T^3 \; \alpha^{8H+3}} \frac{\sum_{k=1}^{N} \theta_k \ln^3(\alpha \theta_k T)} {(\sum_{k=1}^{N} \theta_k)^2}  \quad \text{for }H=\frac{5}{8},\\[10pt]
	           \frac{1}{T^{8-8H} \; \alpha^{8}} \frac{\sum_{k=1}^{N} \theta_k^{8H-4}} {(\sum_{k=1}^{N} \theta_k)^2} \quad \text{for }\frac{5}{8}<H<\frac{3}{4},\\[10pt]
	           \frac{1}{T^{2} \; \alpha^{8}} \frac{\sum_{k=1}^{N} \frac{\theta_k^{2}}{\ln^4(\theta_k T)}} {\left(\sum_{k=1}^{N}\frac{\theta_k}{\ln(\theta_k T)}\right)^2} \quad \text{for }H=\frac{3}{4},\\[10pt]
	           \frac{1}{T^{8-8H} \; \alpha^{8}} \frac{\sum_{k=1}^{N} \theta_k^{8-8H}} {(\sum_{k=1}^{N} \theta_k^{4-4H})^2} \quad \text{for }\frac{3}{4}<H<1.
	       \end{cases}
\end{equation}
Proposition \ref{prop: 4th moment} then yields the bound on the total-variation distance:
\[
d_{TV}\left(\frac{Y_N - \alpha^{-2H}}{\sqrt{\text{var}(Y_N)}},U\right) \leq C \; \sqrt{\zeta(N)}.
\]

Consider now the assumptions \eqref{eq: assumption AS} and show that $\zeta(N) \longrightarrow 0$ with $N \to \infty$. Indeed, in case $H<\frac{3}{4}$ this easily follows from the condition $\theta_k \overset{k\to \infty}{\longrightarrow} \infty$ combined with H\"older inequality.\\
For $H=\frac{3}{4}$, write
\[
\zeta(N) = \frac{1}{T^{2} \; \alpha^{8}} \frac{\sum_{k=1}^{N} \frac{\theta_k^{2}}{\ln^2(\theta_k T)} \frac{1}{\ln^2(\theta_k T)}} {\left(\sum_{k=1}^{N}\frac{\theta_k}{\ln(\theta_k T)}\right)^2}.
\]
Observe that for an increasing positive sequence $\{a_k\}$ and a decreasing positive sequence $\{b_k\}$ the following inequality holds (sometimes referred to as the Chebyshev's sum inequality)
\begin{equation}\label{eq: Chebyshev sum inequality}
\sum_{k=1}^{N}a_k b_k = \sum_{k=1}^{N}a_k \bar{b} + \sum_{k=1}^{N}a_k (b_k - \bar{b}) \leq \bar{b} \sum_{k=1}^{N}a_k,
\end{equation}
where $\bar{b} = \frac{1}{N}\sum_{k=1}^{N} b_k$.\\

Indeed, take $K \in \{1,...,N-1\}$ such that
\begin{equation}
    b_k - \bar{b} > 0 \quad \text{for } k = 1,...,K \quad \text{and} \quad  b_k - \bar{b} \leq 0 \quad \text{for } k = K+1,...,N.
\end{equation}
Since $\{a_k\}$ is increasing and positive, we obtain
\begin{equation}
    \sum_{k=1}^{N}a_k (b_k - \bar{b}) \leq \sum_{k=1}^{K}a_K (b_k - \bar{b}) + \sum_{k=K+1}^{N}a_K (b_k - \bar{b}) = a_K \sum_{k=1}^{N}(b_k - \bar{b}) = 0.
\end{equation}

Now apply \eqref{eq: Chebyshev sum inequality} to the numerator:
\[
\zeta(N) \leq \frac{1}{T^{2} \; \alpha^{8}} \frac{\left(\sum_{k=1}^{N} \frac{\theta_k^{2}}{\ln^2(\theta_k T)} \right) \left(\frac{1}{N}\sum_{k=1}^{N} \frac{1}{\ln^2(\theta_k T)}\right)} {\left(\sum_{k=1}^{N}\frac{\theta_k}{\ln(\theta_k T)}\right)^2}.
\]
H\"older inequality then completes the proof:
\[
\zeta(N) \leq \frac{1}{T^{2} \; \alpha^{8}} \frac{1}{N}\sum_{k=1}^{N} \frac{1}{\ln^2(\theta_k T)} \overset{N\to \infty}{\longrightarrow} 0.
\]
For $H > \frac{3}{4}$, we can prove the convergence of $\zeta(N)$ by direct calculation using $\theta_k \asymp k^\beta$.\\

Having asymptotic normality of $Y_N$, the asymptotic normality of $\alpha^{*}_N$ is now a simple consequence of the delta method with $g(x) = x^{-\frac{1}{2H}}$.
\end{proof}

\begin{Remark}
By Proposition 6.1 in \cite{KrizMaslowski}, for each $K>0$ there exists a constant $C_{K} > 0$ such that we have local Berry-Esseen bound (consider $U \sim \mathcal{N}(0,1)$)
        \begin{equation}\label{eq: BE stat cont}
        \sup_{z\in [-K,K]} \biggl|\mathbb{P}\biggl(\frac{\alpha^{*}_N - \alpha}{\frac{\alpha^{1+2H}}{2H}\sqrt{\text{var}(Y_N)}} \leq z\biggr) - \mathbb{P}\biggl(U \leq z\biggr) \biggr| \leq  C_{K} \sqrt{\zeta(N)},
        \end{equation}
with $ \zeta(N)$ specified in \eqref{eq: def zeta(N)}.
\end{Remark}

\begin{Remark}
In discrete time case, we demonstrated that $\text{var}(Y_N) \asymp \frac{1}{N}$. Thus, continuous-time observations enable us to increase the speed of convergence of $\text{var}(Y_N)$ to zero if the weights are properly modified. This speed is given by \eqref{eq: Y_N asymp variance cont}. Note that this improvement is enabled by the fact that growing $\theta_k$ results in increasing amount information carried by the higher coordinates ($k \to \infty$) due to the corresponding time dilatation.
\end{Remark}

\begin{Remark}\label{rem: nonAN for Y_N}
Note that for $H > \frac{3}{4}$, asymptotic normality was not proved in general. For example if $\theta_k = e^{k}$, $\zeta(N)$ will not converge to zero. In this case the weights grow so rapidly that the highest coordinates dominate in the estimator. This leads to insufficient mixing of (independent) coordinates. Moreover, recall that increasing $\theta_k$ acts (in law) as increasing time horizon (cf. \eqref{eq: self-similarity for integrals}) and the second sample moment of a real-valued fractional Ornstien-Uhlenbeck process with $H>3/4$ converges to the Rosenblatt distribution with increasing time horizon (see e.g. \cite{SebaiyViens2016}). This suggests that one might expect limiting Rosenblatt distribution for $Y_N$ as $N\to \infty$ in case of rapidly growing $\theta_k$ and $H>3/4$. However, detailed investigation of such situation is outside the scope of this article.\\

Interestingly, the estimator \eqref{eq: def rew_MC stat discr} constructed for discrete time observations, converges to normal distribution even in this rapidly-growing $\theta_k$ example, because Theorem \ref{Thm: space asympt - stat discr} does not impose any additional requirements on $\theta_k$. On the other hand, it exhibits lower speed of convergence (in terms of the variance) compared to the continuous-time estimator.\\
\end{Remark}

\section{Estimation in non-stationary case}\label{sec: estimation non-stat}

Space asymptotics of the weighted MCE calculated from a non-stationary solution to equation \eqref{eq: SPDE} with an initial condition \eqref{eq: init_Cond} is studied in this section. Although the construction of the weighted MCE relies on the second moment of the invariant distribution of the stationary solution, the acceleration of virtual time with growing $\theta_k$ (see the self-similarity property \eqref{eq: self-similarity}) can eliminate the effect of a non-stationary initial condition even in fixed time window. Thus, we can expect favorable space-asymptotic properties under some additional assumptions on the growth of $\theta_k$ and $x_k(0)$ (not present in the stationary case).

\subsection{Discrete time observations}
Let $x_k(t), k = 1,\ldots,N$ be the coordinates of a (non-stationary) solution as defined in Definition \ref{def: diag solution} and let these processes are observed in discrete time instants $t=1,\ldots,n$. Consider the weighted minimum-contrast estimator
\begin{equation}\label{eq: rew MC disc - nonStat}
\alpha^{*}_N := \left(\frac{\sum_{k=1}^{N} \frac{\theta_k^{2H}}{\sigma_k^2} \frac{1}{n} \sum_{t=1}^{n}  x_k(t)^2}{N H \Gamma(2H)}\right)^{-\frac{1}{2H}}.
\end{equation}

\begin{Theorem}\label{Thm: space asympt - nonstat discr}
Assume \eqref{eq: stat condition} throughout this theorem.\\
If the following conditions hold:\\
     $(D1) \quad \theta_k \overset{k\to \infty}{\longrightarrow} \infty$, and \\
     $(D2) \quad e^{-2\alpha \theta_k} \frac{\theta_k^{2H}}{\sigma_k^{2}} \mathbb{E}x_k^2(0) \overset{k\to \infty}{\longrightarrow} 0$, \\
then $\alpha^{*}_N$ is \textbf{weakly consistent} in space, i.e. $\alpha^{*}_N  \overset{N\to \infty}{\longrightarrow} \alpha $ in probability.\\

\noindent Let the conditions $(D1), (D2)$ and\\
     $(D3) \quad \sup_{k\in \mathbb{N}} \left(e^{-4\alpha \theta_k} \frac{\theta_k^{4H}}{\sigma_k^{4}} \mathbb{E}x_k^4(0) \right) < \infty$ \\
hold. Then $\alpha^{*}_N$ is \textbf{strongly consistent} in space, i.e. $\alpha^{*}_N  \overset{N\to \infty}{\longrightarrow} \alpha $ almost surely.\\

\noindent Assume there are some constants $C > 0$ and $\beta < -1$ so that\\
     $(D1') \quad e^{-2\alpha \theta_k} < C \; k^{\beta}$, and\\
     $(D2') \quad e^{-2\alpha \theta_k} \frac{\theta_k^{2H}}{\sigma_k^{2}} \mathbb{E}x_k^2(0) < C \; k^{\beta}$.\\
Then $\frac{\alpha^{*}_N - \alpha}{\left(\frac{\alpha}{2H}\sqrt{\frac{2}{n}}\right)\sqrt{\frac{1}{N}}}\overset{N \to \infty}{\longrightarrow} U \sim \mathcal{N}(0,1)$  in distribution.

\end{Theorem}

Observe that $(D1') \Rightarrow (D1)$ and $(D2') \Rightarrow (D2)$.\\

The proof is based on exploring the difference between stationary solutions $z_k(t)$ and non-stationary solutions $x_k(t)$. Conditions $(D\text{x})$ ensure that the difference is negligible for $k \to \infty$ (in an appropriate sense). To simplify the reading process, the detailed proof can be found in Appendix, because it consists of rather technical calculations.

\begin{Remark}\label{rem: non-stat cond adj}
Although the conditions in Theorem \ref{Thm: space asympt - nonstat discr} are rather technical, they are not much restrictive. For example, if
\begin{itemize}
    \item $\frac{\theta_k}{\ln(k)} \overset{k \to \infty}{\longrightarrow} \infty,$
    \item $\inf_k \sigma_k > 0,$ and
    \item $\sup_k \mathbb{E}x^2_k(0) < \infty$,
\end{itemize}
conditions $(D1),(D2),(D1')$ and $(D2')$ are satisfied and $\alpha^{*}_N$ is weakly consistent and asymptotically normal.\\
For strong consistency of $\alpha^{*}_N$ (condition $(D3)$), it suffices to replace $\sup_k \mathbb{E}x^2_k(0) < \infty$ with stronger condition $\sup_k \mathbb{E}x^4_k(0) < \infty$.
\end{Remark}

\subsection{Continuous time observations}

In this section, observation of coordinates of a (non-stationary) solution $x_k(t), k = 1,\ldots,N$ in a fixed time window $t \in [0,T]$ is considered. To let the accelerating time in higher coordinates eliminate the effect of the initial condition, we must leave a certain initial period of time idle (denote its length $\delta > 0$). Define the weighted MCE correspondingly to \eqref{eq: def rew_MC stat cont I}, \eqref{eq: def rew_MC stat cont II} and \eqref{eq: def rew_MC stat cont III}:

\begin{equation}\label{eq: def rew_MC non-stat cont}
\alpha^{*}_N :=
	\begin{cases}
        \left(\frac{\sum_{k=1}^{N} \frac{\theta_k^{2H+1}}{\sigma_k^2} \frac{1}{T-\delta} \int_{\delta}^{T}  x^2_k(t) dt}{ H \Gamma(2H) \sum_{k=1}^{N} \theta_k }\right)^{-\frac{1}{2H}},  \quad \text{for } 0<H<\frac{3}{4},\\

        \left(\frac{\sum_{k=1}^{N} \frac{\theta_k^{\frac{5}{2}}}{\sigma_k^2 \ln(\theta_k T)} \frac{1}{T-\delta} \int_{\delta}^{T}  x^2_k(t) dt}{ \frac{3}{4} \Gamma(\frac{3}{4}) \sum_{k=1}^{N} \frac{\theta_k}{\ln(\theta_k T)} }\right)^{-\frac{1}{2H}}, \quad \text{for }H=\frac{3}{4},\\

        \left(\frac{\sum_{k=1}^{N} \frac{\theta_k^{4-2H}}{\sigma_k^2} \frac{1}{T-\delta} \int_{\delta}^{T}  x^2_k(t) dt}{ H \Gamma(2H) \sum_{k=1}^{N} \theta_k^{4-4H} }\right)^{-\frac{1}{2H}}, \quad \text{for } \frac{3}{4} < H < 1.
    \end{cases}
\end{equation}

\begin{Theorem}\label{Thm: space asympt - nonstat cont}
Assume \eqref{eq: stat condition} throughout this theorem.\\
\noindent If $H \leq \frac{3}{4}$ and\\
     $(C1) \quad \frac{\theta_k}{\ln(k)} \overset{k \to \infty}{\longrightarrow} \infty$, and\\
     $(C2) \quad  \sup_k \left( \frac{\mathbb{E}x^2_k(0)}{\sigma_k^{2}}\right) < \infty$, \\
then $\frac{\alpha^{*}_N -\alpha}{\frac{\alpha^{1+2H}}{2H}\sqrt{\text{var}(Y_N)}} \overset{N \to \infty}{\longrightarrow} U \sim \mathcal{N}(0,1)$  in distribution.\\

\noindent If $H > \frac{3}{4}$ and\\
     $(C1') \quad \theta_k \asymp k^\beta$ for some $\beta > 0$, and\\
     $(C2) \quad  \sup_k \left( \frac{\mathbb{E}x^2_k(0)}{\sigma_k^{2}}\right) < \infty$, \\
then $\frac{\alpha^{*}_N -\alpha}{\frac{\alpha^{1+2H}}{2H}\sqrt{\text{var}(Y_N)}} \overset{N \to \infty}{\longrightarrow} U \sim \mathcal{N}(0,1)$  in distribution.\\

\noindent Let (C1) and\\
     $(C2') \quad  \sup_k \left( \frac{\mathbb{E}x^4_k(0)}{\sigma_k^{4}}\right) < \infty$ \\
hold. Then $\alpha^{*}_N  \overset{N\to \infty}{\longrightarrow} \alpha $ almost surely for all $H \in (0,1)$.
\end{Theorem}

Note that $\text{var}(Y_N)$, whose square root determines the speed of convergence of the estimators, is specified in \eqref{eq: Y_N asymp variance cont} and  $(C1') \Rightarrow (C1)$ and $(C2') \Rightarrow (C2)$.\\

For the sake of simplicity and readability, the technical proof of the theorem is shifted to Appendix.

\begin{Example}\label{example: stoch heat eqn}
The performance of the weighted MCE can be illustrated on the stochastic heat equation on $d$-dimensional domain, with distributed fractional noise, Dirichlet boundary condition and deterministic initial condition, as introduced in Example \ref{example: stoch heat eqn - existence}. It can be interpreted as a diagonalizable stochastic evolution equation with eigenvalues $\theta_k \asymp k^{\frac{2}{d}}$ and $\sigma_k = 1$ for $k = 1,2,\dots$.\\

If first $N$ coordinate projections of the solution (see Definition \ref{def: diag solution}) in discrete time-instants are observed, the weighted minimum-contrast estimator in the form \eqref{eq: rew MC disc - nonStat} can be used. Theorem \ref{Thm: space asympt - nonstat discr} (see Remark \ref{rem: non-stat cond adj} for verification of its assumptions) provides the strong consistency and the asymptotic normality (as $N \to \infty$) of the estimator with the rate of convergence $\frac{1}{\sqrt{N}}$.\\

Next, consider the observations of first $N$ coordinates in continuous time-window $t \in [0,T]$ are available. Since $\theta_k \to \infty$, the continuous-time version of weighted MCE (see \eqref{eq: def rew_MC non-stat cont}) can be applied. Because all conditions in Theorem \ref{Thm: space asympt - nonstat cont} hold, the estimator is strongly consistent and asymptotically normal (as $N \to \infty$) with the rate of convergence

\begin{equation} \label{eq: asymptotics in example}
    \sqrt{\text{var}(Y_N)} \asymp
	       \begin{cases}
        	\frac{1}{\sqrt{N^{1+\frac{2}{d}}}} \quad \text{for }0<H<\frac{3}{4},\\[10pt]
	        \frac{\sqrt{\ln N}}{\sqrt{N^{1+\frac{2}{d}}}} \quad \text{for }H=\frac{3}{4},\\[10pt]
          \frac{1}{\sqrt{N^{1+ \frac{8-8H}{d}}}} \quad \text{for }\frac{3}{4}<H<1.
	       \end{cases}
\end{equation}

Asymptotic formulas for $0<H<3/4$ and $3/4<H<1$ in \eqref{eq: asymptotics in example} are direct applications of \eqref{eq: Y_N asymp variance cont} and integral comparison for series. In case $H=3/4$, start with integral comparison and trivial substitution
\begin{equation}
\sum_{k=1}^{N}\frac{k^{\frac{2}{d}}}{\ln(k^{\frac{2}{d}} T)} \asymp \sum_{k=2}^{N}\frac{k^{\frac{2}{d}}}{\ln(k)} \asymp \int_{e}^{N}\frac{x^{\frac{2}{d}}}{\ln(x)} dx = \int_{1+2/d}^{(1+2/d)\ln(N)} e^y y^{-1} dy.
\end{equation}
To conclude, apply Lemma 2.2 in \cite{Cheridito_et_al2003} with $x=(1+2/d)\ln(N)$ and $\beta = -1$ (or use directly one-step integration-by-parts formula) to get the following asymptotic behavior
\begin{equation}
\int_{1+2/d}^{(1+2/d)\ln(N)} e^y y^{-1} dy \asymp e^{(1+2/d)\ln(N)} [(1+2/d)\ln(N)]^{-1} \asymp \frac{N^{1+2/d}}{\ln N}.
\end{equation}
Asymptotic formula for variance is now an obvious consequence of \eqref{eq: Y_N asymp variance cont}.\\

Speed of convergence of $\alpha^{*}_N$ to $\alpha$ in case of the continuous-time weighted MCE is obviously faster compared to its discrete-time version.
\end{Example}

\section{Comparison to other estimators}\label{sec: comparison}

\subsection{Minimum-contrast estimator}
Recall the standard (non-weighted) minimum-contrast estimator, defined in \cite{Kosky-Loges} and further studied in \cite{KrizMaslowski} and \cite{MaPo-Ergo}.
In diagonalizable case, the estimator can be written as follows:
\[
\hat{\alpha} = (\hat{Y}_\infty)^{-\frac{1}{2H}} =  \left(\frac{\sum_{k=1}^{\infty}  \frac{1}{n} \sum_{t=1}^{n}  x^2_k(t)}{H \Gamma(2H) \sum_{k=1}^{\infty}\frac{\sigma_k^2}{\theta_k^{2H}}}\right)^{-\frac{1}{2H}},
\]
for discrete-time observations and similarly for continuous-time observations. If only first $N$ coordinates are available, the modification is straightforward and it verifies:\\

\[
\text{var}(\hat{Y}_N)=\frac{\sum_{k=1}^{N} \text{var} \left(\frac{1}{n} \sum_{t=1}^{n}  x^2_k(t)\right)}{H^2 \Gamma^2(2H) \left(\sum_{k=1}^{N}\frac{\sigma_k^2}{\theta_k^{2H}}\right)^2}.
\]

If space asymptotics is considered ($N \to \infty$), the numerator of $\text{var}(\hat{Y}_N)$ is growing, whereas the denominator converges to a finite sum. In result, $\text{var}(\hat{Y}_N)$ does not converge to zero with $N \to \infty$ and the estimator is not consistent in space. It was shown in \cite{KrizMaslowski} that this MCE is consistent and asymptotically normal in time (i.e. $n\to \infty$), without assuming diagonality.\\ 

By simple reweighing of the coordinates (the weighted MCE), the poor space-asymptotic properties of the MCE are significantly improved.

\subsection{Maximum likelihood estimator}
Let us start with continuous-time case studied in \cite{Huebner-Rozovskii1995} ($H=1/2$) and later in \cite{Cialenco-Lototsky-Pospisil2009} ($H \geq 1/2$). If considered in the setting of this paper (with $H \geq 1/2$ and $\sigma_k = 1$ for all $k=1,2,\ldots$) the MLE is strongly consistent in space ($N \to \infty$) if and only if
\begin{equation}\label{eq: MLE_cond_cont}
\sum_{k=1}^{\infty} \theta_k = \infty.
\end{equation}
If this holds, the MLE is also asymptotically normal in space (and even asymptotically efficient if $H=1/2$) with speed of convergence given by
\begin{equation}\label{eq: MLE speed of conv}
\frac{1}{\sqrt{\sum_{k=1}^{N}\theta_k}}.
\end{equation}
If compared with the weighted MCE (its speed of convergence is given by the square root of \eqref{eq: Y_N asymp variance cont}), the case $H<1/2$ is covered only by the weighted MCE, in case of $1/2 \leq H < 3/4$, both estimators have the same speed of convergence (which can not be improved if $H=1/2$ due to the asymptotic efficiency of the MLE) and if $3/4 \leq H < 1$, the MLE converges faster than the weighted MCE.\\

In discrete-time case, studied in \cite{Piterbarg-Rozovskii1997} for $H=1/2$, the consistency condition \eqref{eq: MLE_cond_cont} turns into $\sum_{k=1}^{\infty} 1 = \infty$ (in the setting of our paper), which is trivially satisfied. The discrete MLE is thus strongly consistent, asymptotically normal and asymptotically efficient with speed of convergence $1/\sqrt{N}$. The weighted MCE and the MLE have thus the same speed of convergence for $H=1/2$, which can not be improved due to the asymptotic efficiency of the MLE. The author is not aware of any publication which would study the MLE for discrete-time setting in fractional case ($H \neq 1/2$).\\

The implementation of the MLE is rather complicated (for details, see discussion at the end of \cite{Cialenco-Lototsky-Pospisil2009}), in contrast to the simplicity of the weighted MCE. On the other hand, better performance of MLE in case of non-stationary solution can be expected, if only few coordinates are observed.

\subsection{Trajectory fitting estimator}

This estimator was first introduced in \cite{Kutoyants1991} in finite-dimensional setting and recently applied for continuous projections of the solution to diagonalizable parabolic SPDEs driven by a (cylindrical) Wiener process in \cite{Cialenco-Gong-Huang2018}. The explicit expression for the TFE can be found in formula (2.10)  therein and it does not contain any stochastic integration (integration with respect to a random process).\\

If considered in the setting of Example \ref{example: stoch heat eqn - existence} (the heat equation on $d$-dimensional domain with distributed white noise and Dirichlet boundary condition), the TFE (denote $\alpha^{(TFE)}_N$) is strongly consistent in space. Moreover, if $d \geq 2$, it is also asymptotically normal in the following sense
\begin{equation}\label{eq: AN of TFE}
\frac{\alpha^{(TFE)}_N - \alpha + a_N}{b_N} \overset{d}{\longrightarrow} \mathcal{N}(0,1),
\end{equation}
where
\begin{equation}
b_N \asymp \frac{1}{\sqrt{N^{1+\frac{2}{d}}}}, \text{ and } \quad a_N \asymp \frac{1}{N^{\frac{2}{d}}}.
\end{equation}

Note the bias term and its asymptotic behavior $\frac{a_N}{b_N} \asymp N^{\frac{1}{2}-\frac{1}{d}}$. It diverges if $d>2$.\\

In contrast, the weighted MCE is strongly consistent and asymptotically normal without any restriction on the dimension, it has no bias term and the speed of convergence in this example (assuming $H = \frac{1}{2}$) is $\frac{1}{\sqrt{N^{1+\frac{2}{d}}}}.$

\appendix

\section{Proofs from Section \ref{sec: estimation non-stat}}

\textbf{Proof of Theorem \ref{Thm: space asympt - nonstat discr}}
The proof is based on exploring the difference between stationary solutions $z_k(t)$ and non-stationary solutions $x_k(t)$. Denote
\begin{equation}
     Y^{(z)}_N = \frac{\sum_{k=1}^{N} \frac{\theta_k^{2H}}{\sigma_k^2} \frac{1}{n} \sum_{t=1}^{n}  z^2_k(t)}{N H \Gamma(2H)}, \quad Y^{(x)}_N = \frac{\sum_{k=1}^{N} \frac{\theta_k^{2H}}{\sigma_k^2} \frac{1}{n} \sum_{t=1}^{n}  x^2_k(t)}{N H \Gamma(2H)},
\end{equation}
and observe
\[
\mathbb{E}|Y^{(x)}_N - Y^{(z)}_N| \leq \frac{1}{{N H \Gamma(2H)}}\sum_{k=1}^{N} \frac{\theta_k^{2H}}{\sigma_k^2}\frac{1}{n} \sum_{t=1}^{n} \mathbb{E}|x^2_k(t) - z^2_k(t)|.
\]

Clearly
\[x_k(t) - z_k(t) =  e^{-\alpha \theta_k t} (x_k(0) - z_k(0)),\]
and
\begin{equation}
    \begin{aligned}[b]
    &\mathbb{E}|x^2_k(t) - z^2_k(t)| = \mathbb{E}\biggl|\biggl(x_k(t) - z_k(t)\biggr) \biggl(2z_k(t)+(x_k(t) - z_k(t))\biggr)\biggr| \\
    &\leq 2 \sqrt{\mathbb{E}(x_k(t) - z_k(t))^2} \sqrt{\mathbb{E}z^2_k(t)}  + \mathbb{E}(x_k(t) - z_k(t))^2.
    \end{aligned}
\end{equation}
Continue with
\begin{equation}
    \begin{aligned}[b]
    &\mathbb{E}(x_k(t) - z_k(t))^2 \leq e^{-2 \alpha \theta_k }\mathbb{E}(x_k(0) - z_k(0))^2 \\
    &\leq e^{-2 \alpha \theta_k} 4 \left(\mathbb{E}x^2_k(0) + \frac{\sigma_k^2}{(\alpha \theta_k)^{2H}}H \Gamma(2H)\right),
   \end{aligned}
\end{equation}
and
\begin{equation}
    \mathbb{E}z^2_k(t) = \mathbb{E}z^2_k(0) = \frac{\sigma_k^2}{(\alpha \theta_k)^{2H}}H \Gamma(2H).
\end{equation}
These auxiliary calculations yield
\begin{equation}
    \mathbb{E}|Y^{(x)}_N - Y^{(z)}_N| \leq \frac{1}{H \Gamma(2H)} \frac{1}{N}\sum_{k=1}^{N} \biggl( D_k + 2 \sqrt{\frac{H \Gamma(2H)}{\alpha^{2H}}} \sqrt{D_k} \biggr),
\end{equation}
where
\[
D_k = 4 \; e^{-2\alpha \theta_k} \frac{\theta_k^{2H}}{\sigma_k^{2}} \mathbb{E}x_k^2(0) + 4 \; e^{-2\alpha \theta_k} \frac{H \Gamma(2H)}{\alpha^{2H}}.
\]
Conditions $(D1)$ and $(D2)$ guarantee the convergence $D_k \to 0$ and, consequently, $\mathbb{E}|Y^{(x)}_N - Y^{(z)}_N| \overset{N\to \infty}{\longrightarrow}  0$. This, together with the weak consistency of $Y^{(z)}_N$, guarantees the weak consistency of $\alpha^{*}_N$.\\

To prove strong consistency, write
\[
Y^{(x)}_N - Y^{(z)}_N - \mathbb{E}(Y^{(x)}_N - Y^{(z)}_N) = \frac{1}{{N H \Gamma(2H)}}\sum_{k=1}^{N}(Q_k - \mathbb{E}Q_k),
\]
where
\[
Q_k = \frac{\theta_k^{2H}}{\sigma_k^2}\frac{1}{n} \sum_{t=1}^{n} (x^2_k(t) - z^2_k(t)).
\]
By similar calculations, see that $(D1)$ and $(D3)$ imply
\[
\sup_{k\in \mathbb{N}}\biggl( \text{var}(Q_k) \biggr) < \infty.
\]
Kolmogorov SLLN then ensures $Y^{(x)}_N - Y^{(z)}_N - \mathbb{E}(Y^{(x)}_N - Y^{(z)}_N) \overset{N\to \infty}{\longrightarrow} 0$ almost surely. Conditions $(D1)$ and $(D2)$ guarantee $\mathbb{E}(Y^{(x)}_N - Y^{(z)}_N) \to 0$, which leads to
\[
Y^{(x)}_N - Y^{(z)}_N \overset{N\to \infty}{\longrightarrow} 0 \quad \text{almost surely}.
\]
Strong consistency of $\alpha^{*}_N$ now easily follows.\\

Let us conclude with asymptotic normality. Observe
\[
\frac{Y^{(x)}_N - \alpha^{-2H}}{\sqrt{\text{var}(Y^{(z)}_N)}} = \frac{Y^{(z)}_N - \alpha^{-2H}}{\sqrt{\text{var}(Y^{(z)}_N)}} + \frac{Y^{(x)}_N - Y^{(z)}_N}{\sqrt{\text{var}(Y^{(z)}_N)}}.
\]
The first term is asymptotically normal and for the second term, utilize previous calculations and asymptotic behavior of $\text{var}(Y^{(z)}_N)$ specified in \eqref{eq: consist constant} to see:
\begin{equation}
    \begin{aligned}[b]
    &\mathbb{E}\left|\frac{Y^{(x)}_N - Y^{(z)}_N}{\sqrt{\text{var}(Y^{(z)}_N)}}\right| \leq \frac{1}{\sqrt{\text{var}(Y^{(z)}_N)}} \frac{1}{H \Gamma(2H)} \frac{1}{N}\sum_{k=1}^{N} \biggl( D_k + 2 \sqrt{\frac{H \Gamma(2H)}{\alpha^{2H}}} \sqrt{D_k} \biggr) \\
    &\leq C \; \frac{1}{\sqrt{N}}\sum_{k=1}^{N}\sqrt{D_k}.
    \end{aligned}
\end{equation}
Conditions $(D1')$ and $(D2')$ ensure $D_k \leq C \; k^{\beta}$ for some constants $C > 0$ and $\beta < -1$. Hence,
\[
C \; \frac{1}{\sqrt{N}}\sum_{k=1}^{N}\sqrt{D_k} \overset{N \to \infty}{\longrightarrow} 0.
\]
This guarantees asymptotic normality of $Y^{(x)}_N$. Asymptotic normality of $\frac{\alpha^{*}_N - \alpha}{\frac{\alpha^{1+2H}}{2H}\sqrt{\text{var}(Y^{(z)}_N)}}$ then follows from the delta method. Finally, the denominator can be expressed explicitly by applying \eqref{eq: consist constant} again.

\hfill$ \square$


\textbf{Proof of Theorem \ref{Thm: space asympt - nonstat cont}}
Proceed similarly to proof of Theorem \ref{Thm: space asympt - nonstat discr}. Consider
\begin{equation}
Y^{(z)}_N =  \frac{\sum_{k=1}^{N} w_k \frac{1}{T - \delta} \int_{\delta}^{T} z_k(t)^2 dt} {H \Gamma(2H)\sum_{k=1}^{N} w_k \frac{\sigma_k^2}{\theta_k^{2H}}}, \quad   Y^{(x)}_N =  \frac{\sum_{k=1}^{N} w_k \frac{1}{T - \delta} \int_{\delta}^{T} x_k(t)^2 dt} {H \Gamma(2H)\sum_{k=1}^{N} w_k \frac{\sigma_k^2}{\theta_k^{2H}}},
\end{equation}
with weights as in \eqref{eq: def rew_MC non-stat cont}. Employ \eqref{eq: var Y_N cont} to calculate
\[
\mathbb{E}\left|\frac{Y^{(x)}_N - Y^{(z)}_N}{\sqrt{\text{var}(Y_N)}}\right| \leq C \; \frac{1}{\sqrt{\sum_{k=1}^{N} w_k^2 s_k^2}} \sum_{k=1}^{N} \biggl( D_k + \sqrt{w_k \frac{\sigma_k^2}{(\alpha \theta_k)^{2H}}H \Gamma(2H)} \sqrt{D_k} \biggr),
\]
where
\[
D_k = w_k e^{-2 \alpha \theta_k \delta} \left(\mathbb{E}x^2_k(0) + \frac{\sigma_k^2}{(\alpha \theta_k)^{2H}} H\Gamma(2H) \right).
\]
Using formulas for $w_k$ and for asymptotic behavior of $s_k$ (cf. \eqref{eq: def rew_MC stat cont I}, \eqref{eq: def rew_MC stat cont II} and \eqref{eq: def rew_MC stat cont III}) together with condition $(C2)$, gets (in all three cases):
\[
D_k + \sqrt{w_k \frac{\sigma_k^2}{(\alpha \theta_k)^{2H}}H \Gamma(2H)} \sqrt{D_k}  \leq C \; e^{-\frac{\alpha}{2} \theta_k \delta} \quad \text{for some constant } C>0.
\]
Condition $(C1)$ then ensures summability of the corresponding series. Verify further by direct calculation
\[
\sum_{k=1}^{N} w_k^2 s_k^2 \overset{N \to \infty}{\longrightarrow} \infty.
\]
As a consequence,
\[
\mathbb{E}\left|\frac{Y^{(x)}_N - Y^{(z)}_N}{\sqrt{\text{var}(Y_N)}}\right| \overset{N \to \infty}{\longrightarrow} 0.
\]
Asymptotic normality of $\alpha^{*}_N$ follows easily from asymptotic normality of $Y^{(z)}_N$ by the delta method. Note that in case $H> \frac{3}{4}$ the condition $(C1)$ must be strengthen to $(C1')$ to ensure asymptotic normality of $Y^{(z)}_N$ (see Theorem \ref{Thm: space asympt - stat cont}).\\

For strong consistency, write
\[
Y^{(x)}_N - Y^{(z)}_N - \mathbb{E}(Y^{(x)}_N - Y^{(z)}_N) = \frac{1}{\sum_{k=1}^{N} w_k H \Gamma(2H) \frac{\sigma_k^2}{\theta_k^{2H}}} \sum_{k=1}^{N}(Q_k - \mathbb{E}Q_k),
\]
where
\[
Q_k = w_k \frac{1}{T-\delta} \int_{\delta}^{T} (x^2_k(t) - z^2_k(t))dt.
\]
By similar calculations the conditions $(C1)$ and $(C2')$ imply
\[
\sup_{k\in \mathbb{N}}\biggl( \text{var}(Q_k) \biggr) < \infty.
\]
Moreover,
\[
 w_k H \Gamma(2H) \frac{\sigma_k^2}{\theta_k^{2H}} \overset{k \to \infty}{\longrightarrow} \infty.
\]
Hence, application of Kolmogorov SLLN leads to $Y^{(x)}_N - Y^{(z)}_N - \mathbb{E}(Y^{(x)}_N - Y^{(z)}_N) \overset{N \to \infty}{\longrightarrow} 0$ almost surely. The convergence $\mathbb{E}(Y^{(x)}_N - Y^{(z)}_N) \to 0$ follows from the proof of asymptotic normality. In result
\[
Y^{(x)}_N - Y^{(z)}_N \to 0 \quad \text{almost surely}.
\]
Strong consistency of $\alpha^{*}_N$ is a direct consequence.

\hfill$ \square$

\section*{Acknowledgments}
I would like to thank to the anonymous referee for valuable comments and suggestions that helped to improve this paper.

\bibliography{pkbibfile}
\bibliographystyle{plain}

\end{document}